\documentclass[12pt, reqno]{amsart}

\usepackage{amsfonts,amssymb}
\usepackage{hyperref, graphicx}
\usepackage{epsfig}
\usepackage{latexsym}
\usepackage{amsmath,amsthm,calligra,mathrsfs}
\usepackage{mathrsfs}
\usepackage[all,cmtip]{xy}
\usepackage{comment}
\usepackage[dvipsnames]{xcolor}
\usepackage{tikz-cd}
\usepackage{mathrsfs}
\usepackage{bbm}
\usepackage{dynkin-diagrams}
\theoremstyle{plain}
\newtheorem{theorem}{Theorem}[section]
\newtheorem{lemma}[theorem]{Lemma}
\newtheorem{definition}[theorem]{Definition}
\newtheorem{proposition}[theorem]{Proposition}
\newtheorem{cor}[theorem]{Corollary}
\newtheorem{remark}[theorem]{Remark}
\newtheorem{example}[theorem]{Example}

\numberwithin{equation}{section}

\newcommand{\E}{\mathcal{E}}

\newcommand{\Hom}{\mathcal{H}\!\mathit{om}}
\newcommand{\im}{\mathrm{im}}

\begin{document}
\baselineskip=15.5pt

\title{Smooth Relative Connections on Quiver Bundles}
\author{Pavan~Adroja}
\address{Department of Mathematics, IIT Gandhinagar,
 Near Village Palaj, Gandhinagar - 382355, India}
 \email{pavan.a@iitgn.ac.in, adrojapavan@gmail.com}
\author{Sanjay~Amrutiya}
\address{Department of Mathematics, IIT Gandhinagar,
 Near Village Palaj, Gandhinagar - 382355, India}
 \email{samrutiya@iitgn.ac.in}
 \author{Riddhi~Patil}
\address{Department of Mathematics, IIT Gandhinagar,
 Near Village Palaj, Gandhinagar - 382355, India}
 \email{riddhi.patil@iitgn.ac.in}
\thanks{The research work of Pavan Adroja is financially supported by SERB project no. SPON/SERB/62041. The SERB-DST supports SA under project no. CRG/2023/000477. The research work of Riddhi Patil is financially supported by CSIR fellowship under the scheme 09/1031(17071)/2023-EMR-I}
\date{}

\begin{abstract}
We develop a theory of smooth relative connections over the real path algebra $\mathbb{R}Q$ on smooth twisted quiver bundles. We give obstructions to the existence of a smooth relative connection on twisted quiver bundles. For tree-type quiver bundles, we establish a necessary and sufficient condition for the existence of a smooth relative connection. Additionally, we provide a framework for the representation theory of flat quiver bundles, relating the existence of flat relative connections to the underlying quiver representations.
\end{abstract}

\maketitle
\section{Introduction}
Connections on vector bundles are fundamental tools in differential geometry, used to formalise the notion of differentiating sections of vector bundles. They have been studied in various contexts and are central to many geometric constructions. Connections also play a foundational role in theoretical physics, particularly in gauge theory, where they model gauge fields and parallel transport. In this article, we introduce a notion of smooth relative connections over the real path algebra $\mathbb{R}Q$ on smooth twisted quiver bundles, which consist of collections of vector bundles related by morphisms twisted by vector bundles associated with the arrows of a quiver. This framework generalises classical connections within the language of quiver gauge theory. Our definition is motivated by the representation theory of quivers and is designed to be compatible with the underlying quiver structure.

Given any smooth vector bundle $\E$ over a smooth manifold $X$, a smooth connection always exists. However, this is not necessarily the case for twisted quiver bundles due to the presence of additional structure, namely the compatibility conditions imposed by the quiver arrows. In fact, difficulties arise even in simple cases. For example, for quivers of type $\mathbb{A}_n$, finding a smooth relative connection that respects the quiver structure can be nontrivial (see Example \ref{Ex: Non exist}). This obstruction to the existence of a smooth relative connection on twisted quiver bundles is analogous to the holomorphic setting. While smooth connections always exist on smooth vector bundles, holomorphic connections do not always exist on holomorphic vector bundles.

In this article, we study foundational properties of these relative connections and their equivalent characterisations. We also study obstructions to the existence of smooth relative connections on twisted quiver bundles. We provide a necessary and sufficient criterion for the existence of a smooth relative connection on twisted quiver bundles of tree-type, under the assumption that the twisting is trivial. Under the same assumption, we also describe a framework for characterising flat connections in this setting.

We briefly describe the content of the different sections of the paper.

Section 2 reviews the basic definitions of quivers, their representations, the notion of a twisted quiver bundle, and the twisted path algebra sheaf.

Section 3 provides several equivalent characterisations of a smooth relative connection on a smooth twisted quiver bundle. We also establish a cohomological criterion for the existence of such connections, inspired by the framework of \cite[Theorem 4.1]{GK}. To highlight the role of the quiver structure, we include a counterexample demonstrating how compatibility conditions along the arrows can obstruct the existence of a smooth relative connection. Furthermore, we express the existence of smooth relative connections on twisted quiver bundles in terms of the vanishing of an appropriate Atiyah class. This approach is motivated by Atiyah's work \cite{At} on the existence of holomorphic connections on holomorphic vector bundles.

Section 4 states our main result: a necessary and sufficient condition for a smooth twisted quiver bundle of tree-type, under the assumption of trivial twisting, to admit a smooth relative connection. The key technique involves decomposing each vector bundle in the twisted quiver bundle into subbundles that are compatible with the quiver structure.

In Section 5, we generalise the classical correspondence in differential geometry and topology between local systems, flat vector bundles, and representations of the fundamental group. This well-known equivalence connects the topological and differential structures of a space. In our setting, we extend this correspondence to twisted quiver bundles with trivial twisting. Specifically, we establish an equivalence between representations of the fundamental group in the category of quiver representations and smooth twisted quiver bundles with trivial twisting equipped with flat connections.


\section{Preliminaries}
In this section, we recall basic notions on quiver representations. For the basics of quiver representations, see \cite[Chapter~1]{DW}.

A \emph{quiver} is a directed graph given by a quadruple $(Q_0, Q_1, s, t)$, where $Q_0$ is a set of vertices and $Q_1$ is a set of arrows, and $s,t: Q_1\rightarrow Q_0$ are maps. For an arrow $a\in Q_1$, $s(a)$ and $t(a)$ are called the \emph{source} and \emph{target} of $a$, respectively.
	
	A quiver $Q$ is called \emph{finite} if the sets $Q_0$ and $Q_1$ are finite. Throughout this article, we assume the quiver $Q$ to be finite. 
	\begin{definition}
\rm  A \emph{path} $p$ in a quiver $Q$ of length $l \geq 1$ is a sequence $p=a_{l} a_{l-1}\cdots a_{1}$ of arrows in $Q_{1}$ such that $s(a_{i+1})=t(a_{i})$ for $i=1,2,\ldots,l-1$. We define $t(p)=t(a_{l})$ and $s(p)=s(a_{1})$. For every $i \in Q_{0}$, we introduce a trivial path $e_{i}$ of length $0$ and define $s(e_{i})=t(e_{i})=i$.
\end{definition}

	Let $\mathcal{C}$ be any category. A \emph{representation} of a quiver $Q$ in $\mathcal{C}$ is an assignment of an object $V_i$ to each vertex $i\in Q_0$ and a morphism $$\phi_a: V_{s(a)}\rightarrow V_{t(a)}$$
	to each arrow $a\in Q_1$. Given any two representations $(V,\phi)$ and $(W,\psi)$ of quiver $Q$, a \emph{morphism} 
	$$f: (V,\phi) \rightarrow (W,\psi)$$
	is a collection of morphisms $\{f_i: V_i\rightarrow W_i\}_{i\in Q_0}$ in the category $\mathcal{C}$ such that the diagram 
	\[
	\xymatrix{
		V_{s(a)} \ar[d]^{\phi_a} \ar[r]^{f_{s(a)}} & W_{s(a)} \ar[d]^{\psi_a}\\
		V_{t(a)} \ar[r]^{f_{t(a)}} & W_{t(a)}}
	\]
	commutes for all $a\in Q_1$.
	
	\begin{remark}
		\rm    The collection of representations of $Q$ as an object with the morphisms described above forms a category $\mathrm{Rep}(Q,\mathcal{C})$. If the category $\mathcal{C}$ is abelian, then the category $\mathrm{Rep}(Q,\mathcal{C})$ is also abelian.
	\end{remark}
	
	If $\mathcal{C}$ is a category of vector spaces over a field $k$, then $\mathrm{Rep}(Q, \mathcal{C})$ is equivalent to the category $kQ\text{-}\mathrm{mod}$ of finite dimensional left $kQ$-modules, where $kQ$ denotes the path algebra of quiver $Q$ over $k$ (see \cite[Proposition 1.5.4]{DW}). 
	
	Let $X$ be a connected smooth manifold, and let $\mathbf{VB}(X)$ denote the category of vector bundles over $X$. Then, a representation of $Q$ in the category $\mathbf{VB}(X)$ is called \emph{$Q$-bundle} over $X$. That means a $Q$-bundle $\mathcal{R} = (\mathcal{E}, \phi)$ is a collection of vector bundles  $\{\mathcal{E}_{i}\}_{i\in Q_0}$ together with morphisms  $\{\phi_{a}: \mathcal{E}_{s(a)}\rightarrow \mathcal{E}_{t(a)}\}_{a\in Q_1}$. Note that the category of locally free sheaves of $\mathcal{C}^{\infty}_X$-modules is equivalent to the category of smooth vector bundles over $X$, where $\mathcal{C}^{\infty}_X$ denotes the sheaf of smooth functions over $ X$. We will use these notions interchangeably.
	\subsection{Twisted Quiver Bundles} Let $Q$ be a quiver, and let $(X, \mathcal{C}^{\infty}_X)$ be a smooth connected manifold equipped with the sheaf of smooth functions. Fix a collection $M:=\{M_a\}_{a\in Q_1}$ of vector bundles over $X$. An \emph{$M$-twisted $Q$-bundle} is an assignment of a vector bundle $\mathcal{E}_i$ to each vertex $i\in Q_0$ and a morphism $\phi_{a}: \mathcal{E}_{s(a)}\otimes M_a \rightarrow \mathcal{E}_{t(a)}$ to each arrow $a\in Q_1$. 
	Given any two $M$-twisted $Q$-bundles $(\mathcal{E},\phi)$ and $(\mathcal{E}',\phi')$, a \emph{morphism} $f: (\mathcal{E},\phi) \rightarrow (\mathcal{E}', \phi')$ is a collection of morphisms $\{f_i: \mathcal{E}_i\rightarrow \mathcal{E}'_i\}_{i\in Q_0}$ of vector bundles such that the diagram 
	\[
	\xymatrix{
		\mathcal{E}_{s(a)} \otimes M_a \ar[d]^{\phi_a} \ar[rr]^{f_{s(a)}\otimes \mathrm{id}_{M_a}} & & \mathcal{E}'_{s(a)} \otimes M_a \ar[d]^{\psi_a}\\
		\mathcal{E}_{t(a)} \ar[rr]^{f_{t(a)}} & & \mathcal{E}'_{t(a)}}
	\]
	commutes for all $a\in Q_1$. The collection of $M$-twisted $Q$-bundles with the morphisms described above forms a category. If $M_a$ is the trivial line bundle $X\times \mathbb{R}$ for all $a\in Q_1$, then the category of $M$-twisted $Q$-bundles is equivalent to the category of $Q$-bundles by the natural identification $\mathcal{E} \otimes (X\times \mathbb{R})\simeq \mathcal{E}$.

	Let us define the twisted path algebra sheaf $\mathcal{A}$ associated with a quiver $Q$. Set $\mathcal{S}=\bigoplus_{i\in Q_0} \mathcal{C}_X^\infty\cdot e_i$, where $e_i$ are formal symbols indexed by $i\in Q_0$. Equip $\mathcal{S}$ with the structure of a commutative $\mathcal{C}^\infty_X$-algebra by defining $e_i\cdot e_{j}=e_i$ if $i=j$, and $e_i \cdot e_j=0$ otherwise. Let $\mathcal{M}=\bigoplus_{a\in Q_1} M_a$ be a locally free sheaf of $\mathcal{S}$-bimodules. The left $\mathcal{S}$-action is defined by $e_i \cdot m=m$ when $m\in M_a$ and $i=t(a)$, and $e_i\cdot m=0$ otherwise. Similarly, the right action is given by $m\cdot e_i=m$ when $m\in M_a$ and $i=s(a)$, and $m\cdot e_i=0$ otherwise. The \emph{$M$-twisted path algebra} of $Q$ is the tensor $\mathcal{S}$-algebra of the $\mathcal{S}$-bimodule $\mathcal{M}$, that is, 
	$$\mathcal{A}=\bigoplus_{n\geq 0} \mathcal{M}^{\otimes_{\mathcal{S}} n}\,,$$
	where $\mathcal{M}^{\otimes_{\mathcal{S}} n}$ denotes the $n$-fold tensor product of $\mathcal{M}$ over $\mathcal{S}$, with $\mathcal{M}^{\otimes_{\mathcal{S}} 0}=\mathcal{S}$. In particular, $\mathcal{A}$ is naturally a $\mathcal{C}^\infty_X$-algebra.
	
	\begin{remark}\label{QB-M}
		\rm The category of $M$-twisted $Q$-bundles is equivalent to the category of right $\mathcal{A}$-modules that are locally free as $\mathcal{C}^\infty_X$-modules. More generally, this correspondence is established for the category of $M$-twisted $Q$-sheaves (see \cite[Proposition~5.1]{AG}).
	\end{remark}

\section{Smooth connection on twisted quiver bundles}
	
	Let $Q$ be a quiver, and let $(X, \mathcal{C}^{\infty}_X)$ be a smooth connected manifold equipped with the sheaf of smooth functions. Denote by $T_X$ and $T_X^*$ the tangent and cotangent bundles of $X$, respectively. For a smooth vector bundle $\mathcal{E}$ over $X$, let $A^p(\mathcal{E})$ denote the sheaf of smooth $p$-forms on $X$ with values in $\E$.  i.e., $A^p(\mathcal{E})=\bigwedge^{p}T^*_X\, \otimes_{\mathcal{C}^{\infty}_X} \,\mathcal{E} $.
	
	\begin{definition}\label{QC}
		\rm
		A \emph{smooth connection relative over $\mathbb{R}Q$} on an $M$-twisted $Q$-bundle 
		$\mathcal{R} = (\mathcal{E}, \phi)$ over $X$ consists of a collection of smooth connections 
		\[
		\nabla_0 := \{\nabla_i : \mathcal{E}_i \to A^1(\mathcal{E}_i)\}_{i \in Q_0}
		\quad \text{and} \quad
		\nabla_1 := \{\nabla_a : M_a \to A^1(M_a)\}_{a \in Q_1}
		\]
		such that, for every $a \in Q_1$, the diagram
		\[
		\xymatrix{
			\mathcal{E}_{s(a)} \otimes M_a \ar[rr]^{\nabla_{s(a)}\otimes \nabla_a} \ar[d]^{\phi_a} 
			&& A^1(\mathcal{E}_{s(a)}\otimes M_a) \ar[d]^{\mathrm{id}_{T_X^*} \otimes \phi_a} \\
			\mathcal{E}_{t(a)} \ar[rr]^{\nabla_{t(a)}} 
			&& A^1(\mathcal{E}_{t(a)})
		}
		\]
		commutes, where $\nabla_{s(a)} \otimes \nabla_a$ denotes the induced connection on 
		$\mathcal{E}_{s(a)} \otimes M_a$ defined by $\nabla_{s(a)} \otimes \nabla_a:= \nabla_{s(a)}\otimes \mathrm{id}_{M_a}+ \mathrm{id}_{\mathcal{E}_{s(a)}}\otimes \nabla_{a}$. This condition is referred to as the 
		\emph{compatibility} of $\phi_a$ with the connections.
	\end{definition}
	
	In what follows, we only consider connections of the above type. We therefore drop the qualifier "over $\mathbb{R}Q$" and simply refer to them as smooth relative connections. An $M$-twisted $Q$-bundle equipped with the collections of smooth connections $\nabla_0$ and $\nabla_1$ satisfies the compatibility condition for $a \in Q_1$ if and only if the morphism $\phi_{a}$ is parallel with respect to the induced connection on $\Hom(\mathcal{E}_{s(a)}\otimes M_a, \mathcal{E}_{t(a)})$.
	
	\begin{remark}
		\rm
		Let $\mathcal{R}=(\mathcal{E},\phi)$ be an $M$-twisted $Q$-bundle admitting a smooth relative connection. The compatibility condition implies that each morphism
		$\phi_a : \mathcal{E}_{s(a)} \otimes M_a \to \mathcal{E}_{t(a)}$ maps parallel sections of
		$\mathcal{E}_{s(a)} \otimes M_a$ to parallel sections of $\mathcal{E}_{t(a)}$.
	\end{remark}

We introduce the notion of the dual quiver bundle associated with the \emph{opposite quiver} $Q^{op}$, which has the same set of vertices as $Q$ but with all arrows reversed. Let $\mathcal{R}=(\mathcal{E}, \phi)$ be an $M$-twisted $Q$-bundle on $X$. The \emph{dual $M$-twisted $Q^{op}$-bundle}, denoted by $\mathcal{R}^*=(\mathcal{E}^*, \phi^*)$, is defined by assigning the dual vector bundle $\mathcal{E}_i^*$ to each vertex $i\in Q_0$, and the dual morphism $\phi_a^*: \mathcal{E}_{t(a)}^* \otimes M_a \rightarrow \mathcal{E}_{s(a)}^*$  corresponding to $\phi_a: \mathcal{E}_{s(a)} \otimes M_a \rightarrow \mathcal{E}_{t(a)}$ to each arrow $a\in Q_1$ by the tensor-Hom adjunction. Thus, $\mathcal{R}^*$ naturally has the structure of an $M$-twisted $Q^{op}$-bundle over $X$.

\begin{remark}
	\rm	An \(M\)-twisted $Q$-bundle \(\mathcal{R}=(\mathcal{E},\phi)\) admits a smooth relative connection if and only if its dual $M$-twisted $Q^{op}$-bundle \(\mathcal{R}^*=(\mathcal{E}^*,\phi^*)\) admits a smooth relative connection. Indeed, if $(\{\nabla_{i}\}_{i\in Q_0}, \{\nabla_a\}_{a\in Q_1})$ is a smooth relative connection on \(\mathcal{R}\), then the induced dual connections \(\nabla_i^*\) on \(\mathcal{E}_i^*\), together with the same connections on \(M_a\), define a smooth relative connection on $\mathcal{R}^*$.
\end{remark}

	\begin{remark}
	\rm Let $\mathcal{R}=(\mathcal{E}, \phi)$ be an $M$-twisted $Q$-bundle on $X$, and fix smooth connections $\{\nabla_a\}_{a\in Q_1}$ on the twisting bundles $M_a$. If $\nabla=(\{\nabla_i\}_{i \in Q_0}, \{\nabla_a\}_{a\in Q_1})$ and $\nabla'=(\{\nabla'_i\}_{i \in Q_0}, \{\nabla_a\}_{a\in Q_1})$ are two smooth relative connections on $\mathcal{R}$ with the same connections on $M_a$, then their difference $\alpha_{i}:=\nabla'_i-\nabla_i$ defines a collection
	$$\alpha :=\{\alpha_i\}_{i\in Q_0}\in \prod_{i\in Q_0}\mathrm{Hom}_{\mathcal{C}^\infty_X}(\mathcal{E}_i, T^*_X \otimes_{\mathcal{C}^{\infty}_X}\mathcal{E}_{i})$$
	satisfying $\alpha_{t(a)}\circ \phi_a=(\mathrm{id}_{T_X^*} \otimes \phi_a) \circ (\alpha_{s(a)} \otimes \mathrm{id}_{M_a})$ for all $a\in Q_1$. Conversely, given a smooth relative connection $\nabla=(\{\nabla_i\}_{i \in Q_0}, \{\nabla_a\}_{a\in Q_1})$, any such collection $\{\alpha_{i}\}$ determines another smooth relative connection $(\{\nabla_i + \alpha_{i}\}_{i \in Q_0}, \{\nabla_a\}_{a\in Q_1})$. Hence, whenever a smooth connection exists, the set of all smooth connections is modeled on the affine subspace of $ \prod_{i\in Q_0}\mathrm{Hom}_{\mathcal{C}^\infty_X}(\mathcal{E}_i, T^*_X \otimes_{\mathcal{C}^{\infty}_X}\mathcal{E}_{i})$ consisting of those $\{\alpha_i\}_{i\in Q_0}$ satisfying  $\alpha_{t(a)}\circ \phi_a=(\mathrm{id}_{T_X^*} \otimes \phi_a) \circ (\alpha_{s(a)} \otimes \mathrm{id}_{M_a})$ for every $a\in Q_1$.
\end{remark}

The following is an example of a twisted quiver bundle which does not admit a smooth relative connection.
	
\begin{example}\label{Ex: Non exist}
		\rm   Consider a quiver $Q$ as follows:
		\[
		\begin{tikzcd}
			1 \arrow[r,"a"] & 2
		\end{tikzcd}
		\]
		Let $X= \mathbb{R}$ be a base manifold and $\mathcal{E}_{1}= X \times \mathbb{R}^{2}=\mathcal{E}_{2}$ be trivial vector bundles of rank $2$ over $X$, and $M_a = X \times \mathbb{R}$. Assign vector bundles $\mathcal{E}_{1}$ to vertex 1 and $\mathcal{E}_{2}$ to vertex 2. Define a vector bundle morphism $\phi_{a}: \mathcal{E}_{1} \rightarrow \mathcal{E}_{2}$ as $\phi_{a}(s)(x)= x s(x)$ under the natural isomorphism $\mathcal{E}_1 \otimes M_a \simeq \E_1$, where 
		\begin{align*}
			s(x) &= \begin{bmatrix}
				s_{1}(x) \\
				s_{2}(x) 
			\end{bmatrix}
		\end{align*}
		is a section of $\mathcal{E}_{1}$. Suppose there exist connection $\nabla_{1}$ on $\E_1$ and $\nabla_{2}$ in $\E_2$ such that the compatibility condition 
		\begin{align}\label{E: comp con}
		(\mathrm{id}_{T^*_X} \otimes \phi_a)	\circ \nabla_{1} = \nabla_{2} \circ \phi_{a}\
		\end{align}
		 holds. Note that any smooth connection on $\E_i$ is of the form $\nabla_{i}(s)= ds+ A_{i}s$ for $i=1,2$, where 
		$A_{1}$=$\begin{pmatrix}
			a_{11} & a_{12}\\
			a_{21} & a_{22}
		\end{pmatrix} \cdot dx$ and $A_{2}$=$\begin{pmatrix}
			b_{11} & b_{12}\\
			b_{21} & b_{22}
		\end{pmatrix} \cdot dx$
		are matrices of $1$-forms. Then, the condition \eqref{E: comp con} implies that 
		\begin{equation*}
			\begin{split}
				1+xb_{11} &=xa_{11}\\
				xb_{12}&=xa_{12}\\
				xb_{21}&=xa_{21}\\
				1+xb_{22} &=xa_{22}
			\end{split}
		\end{equation*}
	Hence, $a_{11}-b_{11}=\frac{1}{x}$ and $a_{22}-b_{22}=\frac{1}{x}$ are singular at $x=0$ which contradicts the requirement that $a_{ij}$ and $b_{ij}$ are smooth functions everywhere. Thus, the given twisted $Q$-bundle does not admit a smooth relative connection.
\end{example}

We can reformulate the notion of smooth relative connection in the category of $\mathcal{A}$-modules that are locally free as $\mathcal{C}^\infty_X$-modules, using \cite[Proposition~5.1]{AG} (see Remark \ref{QB-M}). Let $\mathcal{R}$ be an $M$-twisted $Q$-bundle, and let $\overline{\mathcal{R}}$ denote the corresponding right $\mathcal{A}$-module.
	
	\begin{definition}\label{MC}
		\rm  A \emph{smooth connection relative over $\mathbb{R}Q$} on $\overline{\mathcal{R}}$ consists of a collection of smooth connections $\nabla_{1}=\{\nabla_{a}\}$ on the bundles $M_a$ and an $\mathbb{R}$-linear map
	$$D: \overline{\mathcal{R}}\longrightarrow T_X^* \otimes_{\mathcal{C}^\infty_X} \overline{\mathcal{R}}$$
		satisfying the following conditions: 
		\begin{enumerate}
			\item Leibniz identity:
			$$D(fs)= df \otimes s + f \cdot D(s)\,,$$
			where $s \in \overline{\mathcal{R}}$ and $f \in \mathcal{C}_X^{\infty}$ are local sections.
			\item For every local section $s \in \overline{\mathcal{R}}$ and $a\in \mathcal{A}$,
			$$D(s\cdot a)= D(s)\cdot a + (\mathrm{id}_{T_X^*}\otimes \mu)(\nabla_{\mathcal{A}}(a)\otimes s)\,,$$
			where $\mu$ denotes the right action of $\mathcal{A}$ on $\overline{\mathcal{R}}$ and $\nabla_{\mathcal{A}}$ denotes the connection on $\mathcal{A}$ induced by $\nabla_{1}$.
		\end{enumerate}
	\end{definition}

This definition is equivalent to Definition \ref{QC} because, by the equivalence between $M$-twisted $Q$-bundles and locally free right $\mathcal{A}$-modules, the quiver morphisms are encoded in the right $\mathcal{A}$-action. Condition $(1)$ is the ordinary Leibniz rule over $\mathcal{C}^\infty_X$, while condition $(2)$ expresses compatibility with the $\mathcal{A}$-action via the induced connection on $\mathcal{A}$; when decomposed with respect to the idempotents $e_i$, it recovers exactly the vertex-wise connections and the usual compatibility condition.

\subsection{Obstructions to smooth relative connections on  twisted quiver bundles}	
	Let $\mathcal{R}=(\mathcal{E},\phi)$ be an $M$-twisted $Q$-bundle on $X$. We know that each $\mathcal{E}_{i}$, $i\in Q_0$, and each $M_a$, $a\in Q_1$, admits a smooth connection. The challenge is to choose $\nabla_i$, $i\in Q_0$ and $\nabla_a$, $a\in Q_1$, such that the compatibility condition 
	$$\nabla_{t(a)} \circ \phi_{a} = (\mathrm{id} \otimes \phi_{a}) \circ (\nabla_{s(a)}\otimes \nabla_a)$$
	holds for every arrow $a \in Q_{1}$. In this section, we fix connections on the twisting bundles and, with respect to this choice, give an obstruction to the existence of connections on the bundles $\mathcal{E}_i$ such that the compatibility conditions hold. Throughout this section, we fix a collection of connections $\nabla_{1}=\{\nabla_{a}\}_{a\in Q_1}$ on the twisting bundles.

If $\nabla^0_{i}$ is a smooth connection on $\mathcal{E}_{i}$, then every smooth connection on $\mathcal{E}_i$ is of the form $\nabla_{i}=\nabla^0_{i} + \alpha_{i}$, where $\alpha_{i}\in \mathrm{Hom}_{\mathcal{C}^\infty_X}(\mathcal{E}_i, T^*_X \otimes_{\mathcal{C}^{\infty}_X}\mathcal{E}_{i})$. Therefore, the compatibility condition becomes a constraint on $\alpha_{s(a)}$ and $\alpha_{t(a)}$:
	\begin{equation*}
		\left(\nabla^{0}_{t(a)}+ \alpha_{t(a)}\right) \circ \phi_{a} = \left(\mathrm{id} \otimes \phi_{a}\right) \circ \left( (\nabla^{0}_{s(a)}+ \alpha_{s(a)}) \otimes \nabla_{a} \right)   
	\end{equation*}
	or equivalently,
	\begin{equation*}
		\alpha_{t(a)} \circ \phi_{a} - (\mathrm{id} \otimes \phi_{a}) \circ (\alpha_{s(a)}\otimes \mathrm{id})=(\mathrm{id} \otimes \phi_{a}) \circ (\nabla^{0}_{s(a)} \otimes \nabla_a) - \nabla^{0}_{t(a)} \circ \phi_{a}\,. 
	\end{equation*}
	
	Define a linear map
	\begin{equation}
		L_{\nabla_1}: \bigoplus_{i \in Q_{0}} \mathrm{Hom}_{\mathcal{C}^\infty_X}(\mathcal{E}_i, T^*_X \otimes_{\mathcal{C}^{\infty}_X}\mathcal{E}_{i})\longrightarrow \bigoplus_{a\in Q_{1}} \mathrm{Hom}_{\mathcal{C}^\infty_X}(\mathcal{E}_{s(a)} \otimes M_a, T^*_X \otimes_{\mathcal{C}^{\infty}_X}\mathcal{E}_{t(a)})
	\end{equation}
by
$$\{\alpha_{i}\}_{i \in Q_{0}} \mapsto \{\alpha_{t(a)} \circ \phi_{a} - (\mathrm{id} \otimes \phi_{a}) \circ (\alpha_{s(a)} \otimes \mathrm{id}) \}_{a\in Q_{1}}$$
	
	Let $\beta_{a}:=(\mathrm{id} \otimes \phi_{a}) \circ (\nabla^{0}_{s(a)} \otimes \nabla_a) - \nabla^{0}_{t(a)} \circ \phi_{a}$. Then, $\{\beta_{a}\}_{a\in Q_1} \in \bigoplus_{a\in Q_{1}} \mathrm{Hom}_{\mathcal{C}^\infty_X}(\mathcal{E}_{s(a)} \otimes M_a, T^*_X \otimes_{\mathcal{C}^{\infty}_X}\mathcal{E}_{t(a)})$. We now state the following proposition, where $\mathcal{R}$ and $T_X^* \otimes \mathcal{R}:=\left(\{T^*_X \otimes \mathcal{E}_i\}_{i\in Q_0}, \{\mathrm{id}_{T^*_X}\otimes \phi_a\}_{a\in Q_1}\right)$ are viewed as objects of the category $\mathcal{A}\text{-}\mathrm{mod}$ of twisted $\mathcal{A}$-modules (see Remark~\ref{QB-M}).
	
	\begin{proposition}\label{P: L-map}
		Let $\mathcal{R}=(\E, \phi)$ be an $M$-twisted smooth $Q$-bundle. Then, $\mathcal{R}$ admits a smooth relative connection whose induced connections on $M_a$ are $\nabla_a$ for all $a\in Q_1$ if and only if the image of $\{\beta_a\}_{a\in Q_1}$ vanishes in $\mathrm{Ext}^1_{\mathcal{A}}(\mathcal{R}, T_X^* \otimes \mathcal{R})$.
	\end{proposition}
	\begin{proof}
		Observe that the map $\partial$  defined in \cite[Theorem 4.1]{GK} is analogous to the map $L_{\nabla_1}$. To apply \cite[Theorem 4.1]{GK} in our framework, we replace the algebra $B$ and the modules $V$ and $W$ with the sheaves $\mathcal{A}, \mathcal{R}$ and $T^*_X \otimes \mathcal{R}$, respectively, as previously introduced. By
		\cite[Theorem 4.1]{GK}, the vanishing of the image of $\{\beta_a\}_{a\in Q_1}$ in $\mathrm{Ext}^1_{\mathcal{A}}(\mathcal{R}, T_X^* \otimes \mathcal{R})$ is equivalent to the existence of a smooth relative connection. In particular, if $\mathrm{Ext}^1_{\mathcal{A}}(\mathcal{R}, T_X^* \otimes \mathcal{R})=0$, then $\mathcal{R}$ admits a smooth relative connection.
	\end{proof}
	
	In \cite{At}, Atiyah described the obstruction to the existence of holomorphic connections on holomorphic vector bundles over complex manifolds in terms of a cohomological class, now called the Atiyah class. He also gave an equivalent description of a holomorphic connection in terms of the splitting of the jet sequence. We provide a similar description for the existence of smooth relative connections on twisted smooth quiver bundles.
		
	Let $\mathcal{R}=(\mathcal{E},\phi)$ be an $M$-twisted smooth quiver bundle, and fix a collection of smooth connections $\nabla_1:=\{ \nabla_{a}\}_{a\in Q_1}$ on $\{M_a\}_{a\in Q_1}$. We define the \emph{first jet $M$-twisted quiver bundle $J^1\mathcal{R}$ with respect to $\nabla_{1}$} as follows. For each $i\in Q_0$, set $J^1\mathcal{E}_i:=\mathcal{E}_i \oplus (T^*_{X}\otimes \mathcal{E}_i)$ with the $\mathcal{C}_X^{\infty}$-module structure on $J^1\mathcal{E}_i$ is given by 
	$f\cdot(s,\tau):=(fs,f\tau + df\otimes s)\,,$
	where $f\in \mathcal{C}_X^{\infty},\, s\in \mathcal{E}_i$ and $\tau\in T^*_{X}\otimes \mathcal{E}_{i}$ are local sections. For each $a\in Q_1$, define a morphism $\tilde{\phi}_{a}: \left(\mathcal{E}_{s(a)} \oplus (T^*_{X}\otimes \mathcal{E}_{s(a)})\right)\otimes M_a \rightarrow \mathcal{E}_{t(a)} \oplus \left(T^*_{X}\otimes \mathcal{E}_{t(a)}\right)$ by 
	$$\tilde{\phi}_a\left((s,\tau)\otimes m\right)=\left(\phi_a(s\otimes m), (\mathrm{id}_{T^*_X}\otimes \phi_a)(\tau \otimes m + s\otimes \nabla_a m)\right)\,,$$
	where $s \in \mathcal{E}_{s(a)}$, $\tau \in T_X^* \otimes \mathcal{E}_{s(a)}$ and $m\in M_a$ are local sections. Note that $s\otimes \nabla_am\in T_X^*\otimes \mathcal{E}_{s(a)}\otimes M_a$ under the natural identification $T_X^*\otimes \mathcal{E}_{s(a)}\otimes M_a\simeq \mathcal{E}_{s(a)}\otimes T_X^* \otimes M_a$.  Clearly, the map $\tilde{\phi}_a$ is $\mathcal{C}_X^{\infty}$-linear.

	There are natural morphisms of twisted quiver bundles $\alpha: T^*_X \otimes \mathcal{R} \rightarrow J^1\mathcal{R}$ and $\beta: J^1\mathcal{R} \rightarrow \mathcal{R}$ defined by collections $\{\alpha_i: T^*_X \otimes \mathcal{E}_i \rightarrow J^1\mathcal{E}_i \}_{i\in Q_0}$ and $\{\beta_i: J^1\mathcal{E}_i \rightarrow \mathcal{E}_i\}_{i\in Q_0}$, where $\alpha_i(\tau)=(0,\tau)$ and $\beta_i(s,\tau)=s$ for local sections $s\in \mathcal{E}_i$ and $\tau \in T^*_X \otimes \mathcal{E}_i$. Note that $\tilde{\phi}_a \circ \left(\alpha_{s(a)} \otimes \mathrm{id}_{M_a}\right)=\left(\mathrm{id}_{T^*_X}\otimes \phi_{a}\right)\circ \alpha_{t(a)}$ and $\phi_a \circ \left(\beta_{s(a)} \otimes \mathrm{id}_{M_a}\right)=\beta_{t(a)} \circ \tilde{\phi}_a$ for all $a\in Q_1$. Hence, the maps $\alpha$ and $\beta$ are morphisms of twisted quiver bundles. Moreover, they fit into the short exact sequence
	\begin{align}\label{Q-JetSeq}
		0 \longrightarrow T^*_X \otimes \mathcal{R} \overset{\alpha}{\longrightarrow} J^1\mathcal{R} \overset{\beta}{\longrightarrow}  \mathcal{R} \longrightarrow 0  
	\end{align}
	of $M$-twisted quiver bundles. Applying the functor $\mathrm{Hom}(\mathcal{R},-)$ to this short exact sequence, and using the argument in \cite[Theorem 5.1]{GK}, we have the following long exact sequence of cohomology:
	$$0 \longrightarrow \mathrm{Hom}(\mathcal{R}, T^*_X \otimes \mathcal{R}) \overset{\alpha^*}{\longrightarrow} \mathrm{Hom}(\mathcal{R}, J^1\mathcal{R}) \overset{\beta^*}{\longrightarrow} \mathrm{Hom}(\mathcal{R},\mathcal{R}) \overset{\delta}{\longrightarrow} \mathrm{Ext}^1(\mathcal{R}, T^*_X \otimes \mathcal{R}) \longrightarrow \cdots $$
	The image of $\mathrm{id}\in \mathrm{Hom}(\mathcal{R},\mathcal{R})$ under the connecting map $\delta$ is called the \emph{Atiyah class} of the $M$-twisted quiver bundle $\mathcal{R}$.

	\begin{proposition}\label{P: At Seq}
		Let $\mathcal{R}$ be an $M$-twisted quiver bundle, and fix a collection of connections $\nabla_1:=\{ \nabla_{a}\}_{a\in Q_1}$ on $\{M_a\}_{a\in Q_1}$. Then the following statements are equivalent:
		\begin{enumerate}
			\item The sequence (\ref{Q-JetSeq}) splits.
			\item The $M$-twisted quiver bundle $\mathcal{R}$ admits a smooth relative connection such that the collection of connections on the bundles $M_a$ coincides with $\nabla_1$.
			\item The Atiyah class of $\mathcal{R}$ vanishes.
		\end{enumerate}
	\end{proposition}
	\begin{proof}
		Suppose the short exact sequence (\ref{Q-JetSeq}) splits. Then there exists a collection $\{\gamma_i:\mathcal{E}_i \rightarrow J^1\mathcal{E}_i\}_{i\in Q_0}$ such that $\beta_i \circ \gamma_i = \mathrm{id}_{\mathcal{E}_i}$ for all $i\in Q_0$. Let $p: J^1\mathcal{R} \rightarrow T^*_X \otimes \mathcal{R}$ be the projection map given by the projections $\{p_i: J^1\mathcal{E}_i \rightarrow T^*_X \otimes \mathcal{E}_i\}_{i\in Q_0}$ defined as $p_i(s,\tau)=\tau$. Note that each $p_i$ is $\mathbb{R}$-linear, but not necessarily $\mathcal{C}_X^{\infty}$-linear. Define $\nabla_i:= p_i \circ \gamma_i:\mathcal{E}_i \rightarrow T^*_X \otimes \mathcal{E}_i$. We claim that $\nabla_0=\{\nabla_i\}_{i\in Q_0}$ defines a smooth connection on $\mathcal{R}$. It is easy to see that $\nabla_i$ is a connection on $\mathcal{E}_i$. It remains to verify the compatibility condition. Let $a\in Q_1$. Then, the diagram 
		\[
		\xymatrix{
			\mathcal{E}_{s(a)} \otimes M_a \ar[d]^{\phi_a} \ar[rr]^{\gamma_{s(a) \otimes \mathrm{id}_{M_a}}\hspace{1.5cm}} &  & \left(\mathcal{E}_{s(a)} \oplus (T^*_X \otimes \mathcal{E}_{s(a)})\right)\otimes M_a \ar[d]^{\tilde{\phi}_a} \ar[rr]^{\hspace{1cm} p_{s(a)}\quad} & & T^*_X \otimes \mathcal{E}_{s(a)}\otimes M_a \ar[d]^{\mathrm{id}_{T^*_X}\otimes \phi_a} \\
			\mathcal{E}_{t(a)} \ar[rr]^{\gamma_{t(a)}\hspace{1cm}} &  & \mathcal{E}_{t(a)} \oplus (T^*_X \otimes \mathcal{E}_{t(a)})  \ar[rr]^{\hspace{1cm} p_{t(a)}} & & T^*_X \otimes \mathcal{E}_{t(a)} }
		\]
		commutes since $(\mathrm{id}_{T^*_X} \otimes \phi_a) \circ p_{s(a)}= p_{t(a)} \circ \tilde{\phi}_a$. This proves (1) $\Longrightarrow$ (2).
		
		To prove (2) $\Longrightarrow$ (1), assume that $\mathcal{R}$ admits a smooth relative connection $\nabla_0=\{\nabla_i:\mathcal{E}_i \rightarrow T_X^* \otimes \mathcal{E}_i\}_{i\in Q_0}$. Since $\nabla_0$ is a smooth relative connection of quiver bundles, we have $(\mathrm{id}_{T_X^*}\otimes \phi_a) \circ (\nabla_{s(a)} \otimes \nabla_a)=\nabla_{t(a)} \circ \phi_a$ for all $a\in Q_1$. Define a morphism $\gamma:\mathcal{R}\rightarrow J^1\mathcal{R}$ by the collection $\{\gamma_i=(\mathrm{id},\nabla_i)\}_{i\in Q_0}$. Note that each $\gamma_i$ is $\mathcal{C}_X^{\infty}$-linear and $\tilde{\phi}_a \circ (\gamma_{s(a)} \otimes \mathrm{id}_{M_a}) = \gamma_{t(a)} \circ \phi_a$. Hence, $\gamma$ is a morphism of twisted quiver bundles, and it is easy to see that $\beta \circ \gamma = \mathrm{id}_{\mathcal{R}}$.
		
		The equivalence between the statements (1) and (3) is standard. 
	\end{proof}

    \begin{remark}
    \rm Proposition \ref{P: L-map} measures the obstruction to correcting $\nabla_i^0$ so as to obtain smooth relative connection. On the other hand, Proposition \ref{P: At Seq} gives an intrinsic and choice-independent formulation. The obstruction is encoded by the Atiyah class of the $M$-twisted quiver bundle $\mathcal R$, defined as the extension class of the jet sequence \eqref{Q-JetSeq}. The vanishing of this class is equivalent to the existence of a smooth relative connection.
    \end{remark}
	
	One can observe that the notion of a holomorphic relative connection on the holomorphic twisted quiver bundles can be defined analogously to the smooth case. The analogue of Proposition \ref{P: L-map} and Proposition \ref{P: At Seq} can also be extended to a holomorphic setting. Since smooth relative connections may not always exist on quiver bundles, the existence of a holomorphic relative connection on holomorphic quiver bundles may require imposing more constraints on the quiver bundle.  However, the following lemma provides a class of vector bundles over a compact Riemann surface for which a holomorphic connection exists.

	\begin{definition}[{\cite[Definition 2.4]{EH}}]
		\rm
		A holomorphic vector bundle $\E$ on a complex manifold $X$ is called \emph{finite} (or \emph{Weil finite}, see \cite{W}) if there exist
		distinct polynomials $f,g \in \mathbb{N}[t]$, where $\mathbb{N}[t]$ denotes the set of
		polynomials in one variable with non-negative integer coefficients, such that $f(\E) \simeq g(\E)$.
		Here, multiplication is defined by the tensor product and addition by the direct sum.
	\end{definition}

	\begin{lemma}
		Let $\mathcal{R}=(\mathcal{E}, \phi)$ be an $M$-twisted holomorphic $Q$-bundle on a compact Riemann surface $X$. Suppose that each $\mathcal{E}_i$  and $M_a$ are finite bundles for all $i\in Q_0$ and $a\in Q_1$. Then $\mathcal{R}$ admits a holomorphic relative connection.
	\end{lemma}
	\begin{proof}
		Let $\mathrm{FConn}(X)$ denote the category of finite bundles equipped with finite connections, and let $\mathcal{C}^{\rm N}(X)$ denote the category of finite bundles. By \cite[Theorem 2.15]{EH}, the forgetful functor 
		$$F:\mathrm{FConn}(X) \longrightarrow \mathcal{C}^{\rm N}(X), \quad (V, \nabla) \mapsto V$$
		is a tensor equivalence of categories. Consequently, for each finite bundle $\E_i$
		and $M_a$, there exist connections $\nabla_{i}$ and $\nabla_{a}$, unique up to isomorphism, such that $F\left((\mathcal{E}_i, \nabla_i)\right)=\mathcal{E}_i$ and $F((M_a, \nabla_a))=M_a$. Moreover, for every $a\in Q_1$, the equivalence yields
		$$\mathrm{Hom}_{\mathrm{FConn}(X)}\left((\mathcal{E}_{s(a)} \otimes M_a, \nabla_{s(a)}\otimes \nabla_a), (\mathcal{E}_{t(a)}, \nabla_{t(a)})\right) \simeq  \mathrm{Hom}_{\mathcal{C}^{\rm N}(X)}\left(\mathcal{E}_{s(a)} \otimes M_a, \mathcal{E}_{t(a)}\right)\, .$$ Hence, each morphism $\phi_{a}: \E_{s(a)} \otimes M_a \rightarrow \E_{t(a)}$ lifts to a morphism compatible with the chosen connections. Therefore, $\mathcal{R}$ admits a holomorphic relative connection.
	\end{proof}
	
\subsection{Bianchi identity for quiver bundle}
	Let $\mathcal{R}=(\mathcal{E},\phi)$ be an $M$-twisted smooth quiver bundle admitting a smooth relative connection $\left( \{\nabla_i\}_{i\in Q_0}, \{\nabla_a\}_{a\in Q_1}\right)$. Thus,  for each $a\in Q_1$, the compatibility condition
	\begin{equation*}
		\nabla_{t(a)} \circ \phi_{a} = (\mathrm{id} \otimes \phi_{a}) \circ (\nabla_{s(a)}\otimes \nabla_{a})
	\end{equation*}
	holds. Let $U\subset X$ be an open set over which every $\mathcal{E}_i$ and $M_a$ are trivial. Let $\{e_1, \ldots ,e_r\}$ and $\{f_1, \ldots ,f_s\}$ be a local frame of $\mathcal{E}_{s(a)} \otimes M_a$, and $\mathcal{E}_{t(a)}$  over $U$, respectively. The morphism $\phi_a:\mathcal{E}_{s(a)} \otimes M_a \rightarrow \mathcal{E}_{t(a)}$ is locally written as 
	$$\phi_a(e_j)= \sum_{i=1}^{s}a_{ij} f_i\,,$$
	and we denote by $A_a=(a_{ij})$ the matrix representation of $\phi_{a}$. Let $\omega^{s(a)\otimes a}$ and $\omega^{t(a)}$ be the connection matrices of $\nabla_{s(a)} \otimes \nabla_a$ and $\nabla_{t(a)}$, respectively. Thus, $$\left(\nabla_{s(a)} \otimes \nabla_a\right)(e)=e\omega^{s(a)\otimes a}, \quad \nabla_{t(a)}(f)=f\omega^{t(a)},$$ 
	where $e=(e_1,\ldots, e_r)$ and $f=(f_1,\ldots, f_s)$. Let $s=\sum_{j=1}^{r} s^{j}e_j = e\cdot \overline{s} \in \Gamma(U, \mathcal{E}_{s(a)} \otimes M_a)$ be a local section, where $\overline{s}=(s^1,\ldots , s^r)^t $ be a column vector. Then,
	\begin{equation*}
		\begin{array}{@{}rl@{}}
			\nabla_{t(a)}(\phi_a(s)) & = f (d(A_a\overline{s})+\omega^{t(a)}A_a\overline{s})\\
			
		\end{array} 
	\end{equation*}
	and
	\begin{equation*}
		\begin{array}{@{}rl@{}}
			(\mathrm{id} \otimes \phi_a)(\nabla_{s(a)} \otimes \nabla_a)(s) & = f(A_ad\overline{s}+A_a\omega^{s(a)\otimes a}\overline{s} )\,.  \\
			
		\end{array} 
	\end{equation*}
	By applying the compatibility condition, we can compare the coefficients
    \begin{equation*}
		\begin{array}{@{}rl@{}}
			d(A_a\overline{s})+\omega^{t(a)}A_a\overline{s} &= A_a(d\overline{s}+\omega^{s(a)\otimes a}\overline{s})\\
			(dA_a)\overline{s}+A_ad\overline{s}+ \omega^{t(a)}A_a\overline{s} &= A_ad\overline{s} + A_a \omega^{s(a)\otimes a}\overline{s}\\
			
		\end{array}
	\end{equation*}
	which gives us:
	$$dA_a + \omega^{t(a)}A_a= A_a \omega^{s(a)\otimes a}\,.$$
	Let $\Omega^{s(a)\otimes a}$ and $\Omega^{t(a)}$ be the curvature matrices of the connections $\nabla_{s(a)}\otimes \nabla_a$ on  $\mathcal{E}_{s(a)}\otimes M_a$ and $\nabla_{t(a)}$ on $\mathcal{E}_{t(a)}$, respectively. It follows that
	$$\Omega^{s(a)\otimes a}= d\omega^{s(a)\otimes a}+ \omega^{s(a)\otimes a}\wedge \omega^{s(a)\otimes a}$$
	$$\Omega^{t(a)}= d\omega^{t(a)}+ \omega^{t(a)}\wedge \omega^{t(a)}$$
	Since $dA_a + \omega^{t(a)}A_a= A_a \omega^{s(a)\otimes a}$, we have
	$$d(dA_a + \omega^{t(a)}A_a - A_a \omega^{s(a)\otimes a})=0$$
	$$d \omega^{t(a)}A_a - \omega^{t(a)} \wedge dA_a -dA_a\wedge \omega^{s(a)\otimes a}-A_ad\omega^{s(a)\otimes a} =0$$
	$$d \omega^{t(a)}A_a - \omega^{t(a)}A_a \wedge \omega^{s(a)\otimes a} + \omega^{t(a)} \wedge \omega^{t(a)}A_a- A_a \omega^{s(a)\otimes a} \wedge \omega^{s(a)\otimes a} + \omega^{t(a)}A_a\wedge \omega^{s(a)\otimes a}- A_a d\omega^{s(a)\otimes a}=0$$
	Hence, we get:
	$$(d \omega^{t(a)} + \omega^{t(a)} \wedge \omega^{t(a)})A_a= A_a(\omega^{s(a)\otimes a} \wedge \omega^{s(a)\otimes a} +  d\omega^{s(a)\otimes a})$$
	$$\Omega^{t(a)}A_a=A_a\Omega^{s(a)\otimes a}$$
	Taking the exterior derivative on both sides, we get,
	$$d\Omega^{t(a)}A_a + \Omega^{t(a)}\wedge dA_a=dA_a \wedge \Omega^{s(a)\otimes a}+A_a d\Omega^{s(a)\otimes a}\,.$$
	By using Bianchi identity for the connections $\nabla_{s(a)}\otimes \nabla_a$ on  $\mathcal{E}_{s(a)}\otimes M_a$ and $\nabla_{t(a)}$ on $\mathcal{E}_{t(a)}$, we get
	$$(\Omega^{t(a)}\wedge \omega^{t(a)}- \omega^{t(a)}\wedge \Omega^{t(a)})A_a +  \Omega^{t(a)}\wedge dA_a = dA_a \wedge \Omega^{s(a)\otimes a}+A_a (\Omega^{s(a)\otimes a}\wedge \omega^{s(a)\otimes a}- \omega^{s(a)\otimes a}\wedge \Omega^{s(a)\otimes a})$$
	for all $a\in Q_1$. This is the Bianchi identity for a quiver bundle, which holds for every arrow $a\in Q_1$.

\section{Smooth relative connection on Tree-type quiver bundles}

In this section, we provide necessary and sufficient conditions for a tree-type smooth twisted quiver bundle over a connected smooth manifold $X$, with trivial twisting, to admit a smooth relative connection. Throughout this section, we assume $M_a$ is a trivial line bundle and a connection $\nabla_a$ is an exterior derivative for $a\in Q_1$. With this assumption and by the identification $\mathcal{E}_{s(a)}\otimes M_a \simeq \mathcal{E}_{s(a)}$, the compatible condition becomes $\nabla_{t(a)}\circ \phi_a = (\mathrm{id}_{T^*_X}\otimes \phi_a)\circ \nabla_{s(a)}$. Hence, we omit the term "twisted" and simply write quiver bundle.

Given a directed path $p=a_{i_k}  \cdots a_{i_1}$ with $a_{i_m}\in Q_1$, we get an associated morphism $\phi_p:=\phi_{a_{i_k}} \circ \cdots \circ \phi_{a_{i_1}}: \mathcal{E}_{s(a_{i_1})} \rightarrow  \mathcal{E}_{t(a_{i_k})}$, we call it a \emph{path morphism}.

\begin{proposition}\label{Constant rank}
If a $Q$-bundle $\mathcal{R} = (\mathcal{E}, \phi)$ admits a smooth relative connection, then all path morphisms $\phi_p$ are of constant rank, where 
$p = a_{i_k} \cdots a_{i_1}$, with $a_{i_m} \in Q_1$, 
is a directed path in the quiver.
\end{proposition}

\begin{proof}
Let $\mathcal{R} = (\mathcal{E}, \phi)$ be a $Q$-bundle that admits a smooth relative connection. Let $p=a_{i_k} \cdots a_{i_1}$ be any arbitrary path with $a_{i_m}\in Q_1$. It is clear that the connections $\nabla_{s(p)}$ and $\nabla_{t(p)}$ on $\mathcal{E}_{s(p)}$ and $\mathcal{E}_{t(p)}$ are compatible with the path morphism $\phi_p$. 

Let $x, y \in X$, and let $\gamma: [0,1] \to X$ be a smooth curve such that $\gamma(0)=x$ and $\gamma(1)=y$. These connections define parallel transport maps
\[
p_{\gamma}^{\mathcal{E}_{s(p)}}: (\mathcal{E}_{s(p)})_{x} \rightarrow (\mathcal{E}_{s(p)})_{y}
\quad \text{and} \quad
p_{\gamma}^{\mathcal{E}_{t(p)}}: (\mathcal{E}_{t(p)})_{x} \rightarrow (\mathcal{E}_{t(p)})_{y}
\]
along $\gamma$. These are linear isomorphisms between the fibres.

The compatibility of $\phi_p$ with the connections gives
\begin{equation*}
  p_{\gamma}^{\mathcal{E}_{t(p)}} \circ {\phi_p}_{x}
  =
  {\phi_p}_{y} \circ p_{\gamma}^{\mathcal{E}_{s(p)}},
\end{equation*}
where ${\phi_p}_x$ and ${\phi_p}_y$ denote the induced maps on the fibres over $x$ and $y$, respectively. Since $p_{\gamma}^{\mathcal{E}_{s(p)}}$ and $p_{\gamma}^{\mathcal{E}_{t(p)}}$ are isomorphisms, we obtain
\begin{equation*}
{\phi_p}_y
=
p_{\gamma}^{\mathcal{E}_{t(p)}} \circ {\phi_p}_x \circ
\big(p_{\gamma}^{\mathcal{E}_{s(p)}}\big)^{-1}.
\end{equation*}
Hence,
\[
\operatorname{rank}({\phi_p}_x)=\operatorname{rank}({\phi_p}_y)
\]
for all $x,y \in X$.
\end{proof}

\subsection{Smooth relative connections on $\mathbb{A}_n$-type quiver bundles}\label{S:A_n conn}
We introduce the following notation:
\begin{itemize}
    \item $\phi_{(i,j)}:=\phi_i \circ \phi_{i-1}\circ \cdots \circ \phi_j$
    
\end{itemize}

\begin{proposition}\label{T:A_n case}
Let $\mathcal{R}=(\mathcal{E},\phi)$ be an $\mathbb{A}_n$-type quiver bundle with all arrows oriented in the same direction. Then $\mathcal{R}$ admits a smooth relative connection if and only if every path morphism is of constant rank.
  
\end{proposition}
\begin{proof}
   Consider the following $\mathbb{A}_n$-type quiver bundle:
\[
\xymatrix{
\E_1  \ar[r]^{\phi_1} & \E_2 \ar[r]^{\phi_2} & \cdots \ar[r]^{\phi_{n-2}} & \E_{n-1} \ar[r]^{\phi_{n-1}} & \E_n
}
\]
Suppose that every path morphism in the above quiver bundle has constant rank.

Since each path morphism is of constant rank, its kernel and image are smooth subbundles. Hence, we decompose the bundles $\E_1, \E_2, \ldots , \E_n$ using the kernels and images of the path morphisms. This process yields, for each component in the decomposition, either an isomorphism onto its image or the zero map, where the image is a direct summand in the decomposition of the target bundle.
 If the induced map is an isomorphism, we choose a smooth connection on one component and transport it to the other.
  If a component is not isomorphic to any other component, we choose an arbitrary smooth connection on it.

The constant rank condition gives a filtration
\[
\ker(\phi_1)
\subseteq
\ker(\phi_{(2,1)})
\subseteq
\cdots
\subseteq
\ker(\phi_{(n-1,1)})
\subseteq
\E_1.
\]
Choose smooth complementary subbundles
$\E_1^{(0)}, \E_1^{(1)}, \dots, \E_1^{(n-1)} \subset \E_1$
such that
\[
\ker(\phi_1) = \E_1^{(0)}, \quad \ker(\phi_{(k+1,1)}) = \ker(\phi_{(k,1)}) \oplus \E_1^{(k)}
\quad \text{for } 1 \le k < n-1,
\]
and
\[
\E_1 = \ker(\phi_{(n-1,1)}) \oplus \E_1^{(n-1)}.
\]
Thus,
\[
\E_1 = \bigoplus_{k=0}^{n-1} \E_1^{(k)}.
\]

Similarly, the constant rank condition yields a filtration
\[
\ker(\phi_2)
\subseteq
\ker(\phi_{(3,2)})
\subseteq
\cdots
\subseteq
\ker(\phi_{(n-1,2)})
\subseteq
\E_2.
\]
Choose smooth complementary subbundles
$\E_2^{(0)}, \dots, \E_2^{(n-2)} \subset \E_2$
giving
\[
\E_2 = \bigoplus_{k=0}^{n-2} \E_2^{(k)}.
\]

Since $\phi_1$ is of
constant rank, $\im(\phi_1)$ is a smooth subbundle of $\E_2$. 
Choose a smooth complementary subbundle $C_2$ such that
\[
\E_2 = \im(\phi_1) \oplus C_2.
\]

We refine the above kernel decomposition of $\E_2$ by intersecting each summand with $\im(\phi_1)$ and with $C_2$. 
Thus each summand $\E_2^{(k)}$ decomposes as
\[
\E_2^{(k)}
=
\big(\E_2^{(k)} \cap \im(\phi_1)\big)
\oplus
\big(\E_2^{(k)} \cap C_2\big).
\]
Thus, we have
$$\E_2 = \left( \bigoplus_{k=0}^{n-2} (\E_2^{(k)} \cap \im(\phi_1)) \right) \oplus \left( \bigoplus_{k=0}^{n-2} (\E_2^{(k)} \cap C_2) \right) $$

For each summand in the decomposition of $\E_1$, the induced map under $\phi_1$ is either zero or an isomorphism onto its image, and this image is a direct summand in the decomposition of $\E_2$.

Observe that the last $n-1$ components in the decomposition of $\E_1$ are isomorphic, via $\phi_1$, to the first $n-1$ components in the decomposition of $\E_2$.
 A smooth connection chosen on the last $n-1$ components of $\E_1$ induces a compatible smooth connection on the first $n-1$ components of $\E_2$.
 Choose a smooth connection on $\ker~\phi_1$ and on the last $n-1$ components of $\E_2$, and by adding the respective connections, give compatible smooth connections on $\E_1$ and $\E_2$.

Continuing this procedure for each $\E_i$ and decomposing it into $i(n-i+1)$ components, we obtain a smooth relative connection on the $\mathbb{A}_n$-type quiver bundle with all arrows oriented in the same direction.

The converse follows from Proposition \ref{Constant rank}.
\end{proof}

\subsection{Smooth relative connections on tree-type quiver bundles}\label{tree conn}

Consider an undirected tree with finitely many vertices. Fix a vertex $r$ in this tree and orient all arrows away from $r$. Then $r$ is the unique vertex with no incoming arrows, and every vertex $v \neq r$ has exactly one incoming arrow. There exists a unique directed path from $r$ to $v$. We call such an oriented tree a \emph{rooted directed tree with root $r$}.
A vertex $v \in Q_0$ is called a \emph{leaf} if it has no outgoing arrows.

Let $Q$ be a quiver that is a rooted directed tree with root $r$, and consider a $Q$-bundle $\mathcal{R}=(\mathcal{E},\phi)$. For every leaf $v \in Q_0$, there is a unique path $p_v$ from $r$ to $v$. This path is of type $\mathbb{A}_{n+1}$, where $n$ is the number of arrows in $p_v$.

Fix a leaf $v$, and let
\[
p_v = a_n \cdots a_1
\]
be the unique path from $r$ to $v$. For each $k = 1, \dots, n$, denote by
\[
p_k = a_k \cdots a_1
\]
the initial subpath of length $k$, and write $\phi_{p_k}$ for the associated path morphism. Then the kernels form an increasing sequence of smooth subbundles
\[
\ker(\phi_{p_1}) \subseteq \ker(\phi_{p_2}) \subseteq \cdots \subseteq \ker(\phi_{p_n}) \subseteq \mathcal{E}_r.
\]
Choose smooth complementary subbundles representing the successive quotients. More precisely, choose subbundles
\[
\mathcal{E}_r^{(0)}, \dots, \mathcal{E}_r^{(n)} \subset \mathcal{E}_r
\]
such that
\[
\ker(\phi_{p_1}) = \mathcal{E}_r^{(0)},
\]
\[
\ker(\phi_{p_{k+1}})
=
\ker(\phi_{p_k})
\oplus
\mathcal{E}_r^{(k)}
\quad \text{for } 1 \le k < n,
\]
and
\[
\mathcal{E}_r
=
\ker(\phi_{p_n})
\oplus
\mathcal{E}_r^{(n)}.
\]
This yields a direct sum decomposition
\[
\mathcal{E}_r
=
\bigoplus_{k=0}^{n}
\mathcal{E}_r^{(k)}.
\]
With the help of this root bundle decomposition, we obtain the following theorem.

\begin{proposition}\label{tree}
Let $Q$ be a rooted directed tree with root $r$. A $Q$-bundle $\mathcal{R} = (\mathcal{E}, \phi)$ admits a smooth relative connection if and only if every path morphism $\phi_p$ in $Q$ is of constant rank.
\end{proposition}
\begin{proof}
Suppose that, for every path $p$ in $Q$, the associated morphism $\phi_p$ is of constant rank. Let $v_1, \dots, v_l$ be the leaves of $Q$. 

For each leaf $v_i$, let $p_{v_i} = a_{n_i i} \cdots a_{1 i}$
be the unique path from $r$ to $v_i$. 
For $1 \le k \le n_i$, denote by
$p_{v_i}^{(k)} := a_{k i} \cdots  a_{1 i}$
the initial subpath of length $k$.
As in the $\mathbb{A}_n$ case, the constant rank condition gives a filtration by smooth subbundles
\[
\ker(\phi_{p_{v_i}^{(1)}})
\subseteq
\ker(\phi_{p_{v_i}^{(2)}})
\subseteq
\cdots
\subseteq
\ker(\phi_{p_{v_i}^{(n_i)}}) \subseteq
\mathcal{E}_r\,.
\]

Choose smooth complementary subbundles representing the successive quotients. 
More precisely, choose subbundles
\[
\mathcal{E}_r^{(v_i,0)}, \dots, \mathcal{E}_r^{(v_i,n_i)} \subset \mathcal{E}_r
\]
such that
\[
\ker(\phi_{p_{v_i}^{(1)}}) = \mathcal{E}_r^{(v_i,0)},
\]
\[
\ker(\phi_{p_{v_i}^{(k+1)}})
=
\ker(\phi_{p_{v_i}^{(k)}})
\oplus
\mathcal{E}_r^{(v_i,k)}
\quad \text{for } 1 \le k < n_i,
\]
and
\[
\mathcal{E}_r
=
\ker(\phi_{p_{v_i}^{(n_i)}})
\oplus
\mathcal{E}_r^{(v_i,n_i)}.
\]
This yields a direct sum decomposition
\[
\mathcal{E}_r
=
\bigoplus_{m_i=0}^{n_i}
\mathcal{E}_r^{(v_i,m_i)}.
\]
We then define the final decomposition of $\mathcal{E}_r$ by
\[
\mathcal{E}_r
=
\bigoplus_{(m_1,\dots,m_l)}
\left(
\mathcal{E}_r^{(v_1,m_1)}
\cap
\cdots
\cap
\mathcal{E}_r^{(v_l,m_l)}
\right),
\]
where $0 \le m_i \le n_i$ for each $i$.

Now fix a vertex $u \neq r$. Since $Q$ is a tree, there is a unique path $p_u$ from $r$ to $u$. Write
$p_u = a_q \cdots a_1.$
Then $\im(\phi_{p_u}) \subseteq \mathcal{E}_u$.
For each $j = 1, \dots, q$, define the intermediate subpaths ending at $u$ by
\[
p_u^{(j)} := a_q \cdots a_j,
\]
and set
\[
I_u^{(j)} := \im(\phi_{a_q} \circ \cdots \circ \phi_{a_j})
\subseteq \mathcal{E}_u.
\]
These are precisely the images of the path morphisms corresponding to subpaths of $p_u$ that end at $u$. Since all path morphisms have constant rank, each $I_u^{(j)}$ is a smooth subbundle of $\mathcal{E}_u$, and the inclusions
\[
I_u^{(1)} \subseteq \cdots \subseteq I_u^{(q)} \subseteq \mathcal{E}_u
\]
form a filtration by subbundles. Choose smooth complementary subbundles
\[
J_u^{(0)} \subseteq I_u^{(1)}, \qquad
J_u^{(j)} \subseteq I_u^{(j+1)} \ \text{for } 1 \le j < q, \qquad \text{and} \qquad J_u^{(q)} \subseteq \mathcal{E}_u,
\]
such that
\[
I_u^{(1)} = J_u^{(0)}, \qquad
I_u^{(j+1)} = I_u^{(j)} \oplus J_u^{(j)} \ \text{for } 1 \le j < q, \qquad \text{and} \qquad
\mathcal{E}_u = I_u^{(q)} \oplus J_u^{(q)}.
\]
This yields a decomposition
\[
\mathcal{E}_u
=
\bigoplus_{j=0}^{q}
J_u^{(j)}.
\]

Now let $w_1, \dots, w_s$ be the leaves of the rooted directed subtree with root $u$. We proceed as in the root case. For each leaf $w_i$, let
$p_{w_i} = a_{e_i i} \cdots a_{1 i}$
be the unique path from $u$ to $w_i$. 
For $1 \le k \le e_i$, denote by
\[
p_{w_i}^{(k)} := a_{k i} \cdots a_{1 i}
\]
the initial subpath of length $k$.

Since all path morphisms have constant rank, the kernels
\[
\ker(\phi_{p_{w_i}^{(1)}})
\subseteq
\cdots
\subseteq
\ker(\phi_{p_{w_i}^{(e_i)}})
\]
form a filtration by subbundles of $\mathcal{E}_u$. Choose smooth splittings of this filtration and let 
$\mathcal{E}_u^{(w_i,k)}$ be complementary subbundles representing the successive quotients.
\[
\mathcal{E}_u
=
\bigoplus_{k=0}^{e_i}
\mathcal{E}_u^{(w_i,k)}.
\]
Taking intersections over all leaves of the subtree, we obtain
\[
\mathcal{E}_u
=
\bigoplus_{(k_1,\dots,k_s)}
\left(
\mathcal{E}_u^{(w_1,k_1)}
\cap
\cdots
\cap
\mathcal{E}_u^{(w_s,k_s)}
\right).
\]

Finally, refining this decomposition with the image filtration
\[
I_u^{(1)} \subseteq \cdots \subseteq I_u^{(q)} \subseteq \mathcal{E}_u,
\]
we obtain the final decomposition
\[
\mathcal{E}_u
=
\bigoplus_{j=0}^{q}
\;
\bigoplus_{(k_1,\dots,k_s)}
\left(
\mathcal{E}_u^{(w_1,k_1)}
\cap
\cdots
\cap
\mathcal{E}_u^{(w_s,k_s)}
\cap
J_u^{(j)}
\right).
\]

We now construct a smooth relative connection on the $Q$-bundle.

At the root $r$, consider the refined decomposition
\[
\mathcal{E}_r
=
\bigoplus_{(m_1,\dots,m_l)}
\left(
\mathcal{E}_r^{(v_1,m_1)}
\cap
\cdots
\cap
\mathcal{E}_r^{(v_l,m_l)}
\right).
\]
Choose an arbitrary smooth connection on each summand and denote the
resulting connection on $\mathcal{E}_r$ by $\nabla_{\mathcal{E}_r}$.

Since $Q$ is a rooted directed tree with root $r$, every vertex $v \neq r$ has a unique
incoming arrow. We now define the connection on each vertex by
moving outward from the root along the arrows.

Let $a : u \to v$ be an arrow such that a connection has already
been defined on $\mathcal{E}_u$. Write the refined decomposition of $\mathcal{E}_u$ as
\[
\mathcal{E}_u
=
\bigoplus_{j=0}^{q}
\;
\bigoplus_{(k_1,\dots,k_s)}
\left(
\mathcal{E}_u^{(w_1,k_1)}
\cap
\cdots
\cap
\mathcal{E}_u^{(w_s,k_s)}
\cap
J_u^{(j)}
\right),
\]
where the summands are those obtained from the intersection of
the kernel and image filtrations.

For each such summand
\[
\mathcal{S}_u
=
\left(
\mathcal{E}_u^{(w_1,k_1)}
\cap
\cdots
\cap
\mathcal{E}_u^{(w_s,k_s)}
\cap
J_u^{(j)}
\right),
\]
the restriction
\[
\phi_a\big|_{\mathcal{S}_u}
:
\mathcal{S}_u
\longrightarrow
\mathcal{E}_v
\]
is either zero or an isomorphism onto its image, and this image
is a direct sum of certain summands in the corresponding
decomposition of $\mathcal{E}_v$.

If $\phi_a|_{\mathcal{S}_u} = 0$, we choose an arbitrary smooth
connection on the corresponding summands of $\mathcal{E}_v$. If $\phi_a|_{\mathcal{S}_u}
:
\mathcal{S}_u
\longrightarrow
\phi_a(\mathcal{S}_u)$ is an isomorphism, we define the connection on
$\phi_a(\mathcal{S}_u)$ by :
\[
\nabla_{\phi_a(\mathcal{S}_u)}
=
(\mathrm{id}_{T_X^*}\otimes \phi_a) \circ \nabla_{\mathcal{S}_u} \circ \phi_a^{-1}.
\]
Taking the direct sum over all summands defines a smooth
connection on $\mathcal{E}_v$.

Proceeding along all arrows starting from the root,
this defines smooth connections on every vertex of $Q$.
Since each vertex has a unique incoming arrow,
the construction is well defined.

By construction, for every arrow $a : u \to v$ we have
\[
\nabla_{\mathcal{E}_v} \circ \phi_a
=
(\mathrm{id}_{T^*_X} \otimes \phi_a)
\circ
\nabla_{\mathcal{E}_u}.
\]
Hence, the $Q$-bundle admits a smooth relative connection.
\end{proof}

Up to this point, we have constructed a smooth relative connection on rooted directed tree-type quiver bundles under the assumption that each path morphism has constant rank. A natural question arises: what happens when the arrows are not uniformly oriented? This question is addressed in Proposition~\ref{P:Conn-Q^a}, which asserts that if a quiver bundle admits a smooth relative connection, then one may reverse the orientation of any arrow. The resulting quiver bundle still admits a smooth relative connection.

Let $Q=(Q_0, Q_1)$ be a quiver, and let $a\in Q_1$. Define the quiver $Q^a$ by reversing the arrow $a$ and keeping the other arrow and vertices unchanged. For example, if $Q$ is given as 
\[
\xymatrix{
1 \ar[r]  & 2  & 3 \ar[l] \ar[r]^{a} & 4 \ar[r] & 5  \\
& & &  6 \ar[u] & }
\]
Then, the $Q^a$ is 
\[
\xymatrix{
1 \ar[r]  & 2  & 3 \ar[l]  & 4 \ar[l]_{a} \ar[r] & 5  \\
& & &  6 \ar[u] & }
\]

If $\mathcal{R}=(\mathcal{E},\phi)$ is a $Q$-bundle and $\phi_a:\mathcal{E}_{s(a)} \rightarrow \mathcal{E}_{t(a)}$ is a constant rank morphism, then one can define the corresponding $Q^a$-bundle $\mathcal{R}^a$ as follows: 

Since $\phi_a$ has constant rank, its kernel and image are smooth subbundles. 
Choose smooth complementary subbundles 
$F_{s(a)} \subset \mathcal{E}_{s(a)}$ and 
$F_{t(a)} \subset \mathcal{E}_{t(a)}$ such that
\[
\mathcal{E}_{s(a)} = \ker\phi_a \oplus F_{s(a)},
\qquad
\mathcal{E}_{t(a)} = \im\phi_a \oplus F_{t(a)}.
\]
 Further, $\phi_a$ induces the isomorphism $\tilde{\phi_a}:  F_{s(a)}\rightarrow \im~\phi_a$. The $Q^a$-bundle $\mathcal{R}^a$ is defined by taking the same bundles at all vertices and all arrows except $\phi_a$, replace $\phi_a$ by the morphism $\imath \circ \tilde{\phi_a}^{-1} \circ p: \mathcal{E}_{t(a)} \rightarrow \mathcal{E}_{s(a)}$, where $p: \mathcal{E}_{t(a)} \rightarrow \im~\phi_a$ and $\imath:  F_{s(a)} \rightarrow \mathcal{E}_{s(a)}$ denote the projection and inclusion map, respectively. 

\begin{proposition}\label{P:Conn-Q^a}
    Let $\mathcal{R}=(\mathcal{E},\phi)$ be a $Q$-bundle, and let $a\in Q_1$. Then, the following are equivalent:
    \begin{enumerate}
        \item The $Q$-bundle $\mathcal{R}$ admits a smooth relative connection.
        \item The morphism $\phi_a : \mathcal{E}_{s(a)}\rightarrow \mathcal{E}_{t(a)}$ is of constant rank, and the corresponding $Q^a$-bundle $\mathcal{R}^a$ admits a smooth relative connection.
    \end{enumerate}
\end{proposition}
\begin{proof}
Assume that the $Q$-bundle $\mathcal{R}=(\mathcal{E},\phi)$ admits a smooth relative connection. Clearly, the morphism $\phi_a$ is of constant rank by Proposition \ref{Constant rank}. From the construction of $\mathcal{R}^a$, it is enough to prove that the connection $\nabla_{\mathcal{E}_{s(a)}}$ on $\mathcal{E}_{s(a)}$ and $\nabla_{\mathcal{E}_{t(a)}}$ on $\mathcal{E}_{t(a)}$ are compatible with the morphism $\imath \circ \tilde{\phi_a}^{-1} \circ p : \mathcal{E}_{t(a)} \rightarrow \mathcal{E}_{s(a)}$. The compatibility of the connections with $\phi_a$ gives the induced connections on $\ker~\phi_a$ and $\im~\phi_a$. Further, that decomposes the connection
    $$\nabla_{\mathcal{E}_{s(a)}}=\nabla_{\ker~\phi_a}+\nabla_{F_{s(a)}}$$
    and 
    $$\nabla_{\mathcal{E}_{t(a)}}=\nabla_{\im~\phi_a}+\nabla_{F_{t(a)}}\,.$$
In this way, we have a connection on each component $\ker~\phi_a,\, F_{s(a)},\, \im~\phi_a$ and $F_{t(a)}$. These connections are compatible with the map $p,\, \imath$ and the isomorphism $\tilde{\phi_a}$ induced by $\phi_a$. i.e. we have
$$\nabla_{\im~\phi_a} \circ \tilde{\phi_a}= (\tilde{\phi_a}\otimes \mathrm{id})\circ \nabla_{F_{s(a)}}\,.$$
That implies: 
$$(\tilde{\phi_a}\otimes \mathrm{id})^{-1} \circ \nabla_{\im~\phi_a} \circ \tilde{\phi_a} \circ \tilde{\phi_a}^{-1}= (\tilde{\phi_a}\otimes \mathrm{id})^{-1} \circ (\tilde{\phi_a}\otimes \mathrm{id})\circ \nabla_{F_{s(a)}} \circ \tilde{\phi_a}^{-1}$$
$$(\tilde{\phi_a}^{-1} \otimes \mathrm{id}) \circ \nabla_{\im(\phi_a)}= \nabla_{F_{s(a)}} \circ \tilde{\phi_a}^{-1}$$
Hence, the $Q^a$-bundle $\mathcal{R}^a$ admits a smooth relative connection.

Conversely,  suppose $\mathcal{R}^a$ admits a smooth relative connection. Note that $(Q^a)^a=Q$ and $(\mathcal{R}^a)^a\simeq \mathcal{R}$. Then, this assertion is followed by similar arguments as above. 
\end{proof}

Given a tree-type $Q$-bundle $\mathcal{R}=(\mathcal{E},\phi)$ with all morphisms $\phi_a$ of constant rank, we obtain a rooted directed tree-type quiver bundle after successive reversal of a finite set of arrows $a_1,\cdots,a_m$. We denote it by $\mathcal{R}^{(a_1,\cdots,a_m)}$.

\begin{theorem}
Let $\mathcal{R}=(\mathcal{E},\phi)$ be a tree-type quiver bundle. $\mathcal{R}$ admits a smooth relative connection if and only if all morphisms $\phi_a$ are of constant rank and all path morphisms of $\mathcal{R}^{(a_1,\cdots,a_m)}$ are of constant rank.
\end{theorem}
\begin{proof}
The result follows from Proposition~\ref{tree} together with Proposition~\ref{P:Conn-Q^a}.
\end{proof}


\section{Representations and flat connection}

Let $X$ be a smooth, connected manifold, and let $x\in X$ be fixed. Let $\mathbf{VB}(X)^{\rm F}$ denote the category of flat vector bundles over $X$ and $\mathrm{Rep}(\pi_1(X,x), \mathbf{VSpace})$ denote the category of representations of $\pi_1(X,x)$ in the category $\mathbf{VSpace}$ of finite dimensional $\mathbb{R}$-vector spaces.  We recall the equivalence between the categories  $\mathbf{VB}(X)^{\rm F}$ and $\mathrm{Rep}(\pi_1(X,x), \mathbf{VSpace})$.

Let $\E$ be a smooth vector bundle of rank $n$ on $X$ with a flat connection $\nabla$. Define a representation 
$$\rho_{(\E,\nabla)}:\pi_1(X,x)\longrightarrow \mathrm{GL}(\E_x)$$
by mapping each loop to the automorphism of the fiber given by parallel transport along the loop. Since the connection $\nabla$ is flat, $\rho_{(\E,\nabla)}$ is well-defined. The map $\rho_{(\E,\nabla)}$ is called \emph{monodromy representation} corresponding to $(\E,\nabla)$. Conversely, given a representation 
$$\rho:\pi_1(X,x)\longrightarrow \mathrm{GL}(V)\,,$$
there is an associated vector bundle $\E_{\rho}:= (\tilde{X} \times V)/\pi_1(X,x)$, where the action of $\pi_1(X,x)$ on $\tilde{X}$ is given by the Deck transformation and on $V$ by the representation $\rho$. Note that the pull-back of $\E_{\rho}$ is trivial over the universal covering map $\tilde{X}\rightarrow X$. This trivial bundle has a natural flat connection given by the exterior derivatives, which descends the connection on $\E_{\rho}$. In this way, given a representation, we have a vector bundle with a flat connection. 

In this way, we have a correspondence between vector bundles admitting a flat connection and representations of the fundamental group. This correspondence is functorial and gives an equivalence between these two categories. It is known as the \emph{Riemann-Hilbert correspondence} in the smooth category and also holds in the holomorphic setup.

Let $(\E_1,\nabla_1)$ and $(\E_2,\nabla_2)$ be two flat bundles and $\phi:(\E_1,\nabla_1)\rightarrow (\E_2,\nabla_2)$ be a morphism of flat bundles, i.e., morphism of vector bundles which is compatible with the flat connections $\nabla_1$ and $\nabla_2$. The map $\phi$ induces a map $\rho_{\phi}: \rho_{(\E_1,\nabla_1)}\rightarrow \rho_{(\E_2,\nabla_2)}$, i.e., for each $g\in \pi_1(X,x)$, there is a homomorphism $\mathrm{GL}({\E_1}_x)\rightarrow \mathrm{GL}({\E_2}_x)$ which is compatible with the map $\phi_x:{\E_1}_x \rightarrow {\E_2}_x$. In other words, the map $\phi$ can be regarded as a flat $\mathbb{A}_2$-type quiver bundle which is assigned to a representation of $\pi_1(X,x)$ in the category of $\mathbb{A}_2$-type quiver vector spaces. This is true more generally for any quiver (see Corollary \ref{cor}).

Let $G$ be a group and $\mathcal{C}$ be an abelian category. Consider the category $\mathrm{Rep}(G, \mathcal{C})$ with the following data:

\begin{itemize}
    \item Objects are pairs $(V, \rho)$, where $V \in \mathrm{Ob}(\mathcal{C})$ and $\rho: G\rightarrow \mathrm{Iso}(V)$ is a homomorphism. The set $\mathrm{Iso}(V)$ denotes the group under the composition of isomorphisms from $V$ to $V$.
    \item Let $(V_1,\rho_1)$ and $(V_2,\rho_2)$ be two objects of $\mathrm{Rep}(G,\mathcal{C})$. The morphism $f:(V_1,\rho_1) \rightarrow (V_2,\rho_2)$ is a morphism $f: V_1 \rightarrow V_2 \in \mathrm{Hom}_{\mathcal{C}}(V_1,V_2)$ such that $f\circ \rho_1(g) = \rho_2(g) \circ f$, for all $g\in G$.
\end{itemize}

The category $\mathrm{Rep}(G,\mathcal{C})$ is an abelian category. We want to prove an equivalence between the categories $\mathrm{Rep}(Q,\mathrm{Rep}(G,\mathcal{C}))$ and $\mathrm{Rep}(G,\mathrm{Rep}(Q,\mathcal{C}))$. Let $A=((V,\rho),\phi)$ be an object of $\mathrm{Rep}(Q,\mathrm{Rep}(G,\mathcal{C}))$, i.e., there is a collection $(V,\rho)=\{(V_i,\rho_i)\}_{i\in Q_0}$ and $\phi=\{\phi_a: (V_{s(a)},\rho_{s(a)})\rightarrow (V_{t(a)},\rho_{t(a)})\}_{a\in Q_1}$, where $(V_i,\rho_i)\in \mathrm{Ob}(\mathrm{Rep}(G,\mathcal{C}))$ and $\phi_a \in \mathrm{Mor}(\mathrm{Rep}(G,\mathcal{C}))$. We will define the corresponding object $F(A)$ of $A$ in the category $\mathrm{Rep}(G,\mathrm{Rep}(Q,\mathcal{C}))$. Consider the object $\mathbb{V}:=(\{V_i\}_{i\in Q_0},\{\phi_a:V_{s(a)}\rightarrow V_{t(a)}\}_{a\in Q_1})$ of $\mathrm{Rep}(Q,\mathcal{C})$, where $\phi_a$ is a morphism in $\mathcal{C}$ via the natural inclusion. Define $F(A):= (\mathbb{V},\rho)$, where $\rho: G\rightarrow \mathrm{Iso}(\mathbb{V})$ is a homomorphism given by $\rho=(\rho_i:G\rightarrow \mathrm{Iso}(V_i))$. Let $\Psi: ((V,\rho),\phi) \rightarrow ((W,\rho'),\phi')$ be a morphism in the category $\mathrm{Rep}(Q,\mathrm{Rep}(G,\mathcal{C}))$, i.e., $\Psi$ is a collection $\{ \Psi_{i}: (V_i,\rho_i)\rightarrow (W_i,\rho_i') \}_{i\in Q_0}$ of morphisms which is compactible with $\phi_a$ and $\phi_a'$ for all $a\in Q_1$. We define the associated morphism $F(\Psi): (\mathbb{V},\rho)\rightarrow (\mathbb{W},\rho')$ in the category $\mathrm{Rep}(G,\mathrm{Rep}(Q,\mathcal{C}))$, i.e., we need to specify a morphism $\Phi:\mathbb{V}\rightarrow \mathbb{W}$ which is compatible with $\rho$ and $\rho'$. The collection $\{\Psi_i \}_{i\in Q_0}$ will define the morphism $\Phi$ and it is compatible as $\Psi_i$'s are compatible with $\phi_a$ and $\phi_a'$. In this way, we have a functor 
$$F:\mathrm{Rep}(Q,\mathrm{Rep}(G,\mathcal{C}))\longrightarrow \mathrm{Rep}(G,\mathrm{Rep}(Q,\mathcal{C}))\,.$$
This functor gives an equivalence of categories. The inverse functor can be given as follows. Let $B:=((V,\phi),\rho)$ be an object of $\mathrm{Rep}(G, \mathrm{Rep}(Q,\mathcal{C}))$. We want to define an associated object $G(B)$ in the category $\mathrm{Rep}(Q,\mathrm{Rep}(G,\mathcal{C}))$. The collection $\{(V_i,\rho_i)\}_{i\in Q_0}$ with the morphisms $\{\phi_a\}_{a\in Q_1}$ gives an object of the category $\mathrm{Rep}(Q,\mathrm{Rep}(G,\mathcal{C}))$. Let $\Theta: ((V,\phi),\rho)\rightarrow ((W,\phi'),\rho')$ be a morphism in the category $\mathrm{Rep}(G,\mathrm{Rep}(Q,\mathcal{C}))$. This has a collection of morphisms $\{ \Theta_i: (V_i,\phi)\rightarrow (W_i,\phi')\}_{i\in Q_0}$ which is compatible with $\rho$ and $\rho'$. Define $G(\Theta): ((V,\rho),\phi)\rightarrow ((W,\rho'),\phi')$ by using the $\Theta_i$ at each vertices. This gives morphism in the category $\mathrm{Rep}(Q,\mathrm{Rep}(G,\mathcal{C}))$. In this way, we have a functor 
$$G:\mathrm{Rep}(G,\mathrm{Rep}(Q,\mathcal{C}))\longrightarrow \mathrm{Rep}(Q,\mathrm{Rep}(G,\mathcal{C}))$$
which is the inverse of the functor $F$. Hence, we have the following proposition. 

\begin{proposition}\label{Rep=Flatconn}
    The following categories are equivalent:
    $$\mathrm{Rep}(Q,\mathrm{Rep}(G,\mathcal{C})) \simeq \mathrm{Rep}(G,\mathrm{Rep}(Q,\mathcal{C}))$$
\end{proposition}

\begin{definition}
\rm    A $Q$-bundle $\mathcal{R}=(\E,\phi)$ is called a \emph{flat $Q$-bundle} if it admits a compatible flat connection at each vertex. 
\end{definition}

The category of flat $Q$-bundles is equivalent to the category
$\mathrm{Rep}(Q,\mathbf{VB}(X)^{\rm F})$.
 As mentioned in the beginning of this section, we have an equivalence between the categories $\mathrm{Rep}(\pi_1(X), \mathbf{VSpace})$ and $\mathbf{VB}(X)^{\rm F}$. By using this, one can identify the category of quiver bundles having a flat connection with the category $\mathrm{Rep}(Q,\mathrm{Rep}(\pi_1(X),\mathbf{VSpace}))$. By Proposition \ref{Rep=Flatconn}, we get a condition when a given bundle admits a flat connection.

\begin{cor}\label{cor}
    The following categories are equivalent:
    $$\mathrm{Rep}(Q,\mathbf{VB}(X)^{\rm F})\simeq \mathrm{Rep}(\pi_1(X,x),\mathrm{Rep}(Q,\mathbf{VSpace}))$$
    i.e. the category of flat quiver bundles over 
$X$ is equivalent to the category of representations of 
$\pi_1(X,x)$ into the category of quiver representations.
\end{cor}




\end{document}